\documentclass{amsart}

\usepackage{mathtools}

\usepackage{amssymb}

\usepackage[colorlinks,linkcolor=blue,bookmarksnumbered,pagebackref,final]{hyperref}

\newtheorem{theorem}{Theorem}[section]
\newtheorem{lemma}[theorem]{Lemma}
\newtheorem{corollary}[theorem]{Corollary}

\newtheorem{proposition}[theorem]{Proposition}

\theoremstyle{definition}
\newtheorem{definition}[theorem]{Definition}

\theoremstyle{remark}
\newtheorem{remark}[theorem]{Remark}

\numberwithin{equation}{section}

\newcommand{\ahg}{\hat{\mathsf{A}}}
\newcommand{\End}[1]{\operatorname{End}(#1)}
\newcommand{\aset}[1]{\{#1\}}

\newcommand{\rK}{\mathrm{K}}
\newcommand{\rT}{\mathrm{T}}
\newcommand{\bS}{\mathbb{S}}
\newcommand{\bC}{\mathbb{C}}
\newcommand{\bQ}{\mathbb{Q}}
\newcommand{\bZ}{\mathbb{Z}}
\newcommand{\cC}{\mathcal{C}}
\newcommand{\cV}{\mathcal{V}}
\newcommand{\bfL}{\mathbf{L}}

\DeclareMathOperator{\kcw}{K-cw_2}
\DeclareMathOperator{\acw}{\Hat{A}-cw_2}
\DeclareMathOperator{\opc}{c}
\DeclareMathOperator{\opR}{R}
\DeclareMathOperator{\ch}{ch}
\DeclareMathOperator{\id}{id}

\begin{document}

\title[A relation between two cowaists]{On a relation between the $\mathrm{K}$-cowaist \\ and the $\hat{\mathsf{A}}$-cowaist}

\author[X. Wang]{Xiangsheng Wang}
\address{School of Mathematics, Shandong University, Jinan, Shandong 250100, China}
\email{xiangsheng@sdu.edu.cn}
\thanks{The author is partially supported by NSFC Grant No.\ 12101361.}


\date{\today}

\dedicatory{}

\begin{abstract}
  The $\rK$-cowaist $\kcw(M)$ and the $\ahg$-cowaist $\acw(M)$ are two interesting invariants on a manifold $M$, which are closely related to the existence of the positive scalar curvature metric on $M$. In this note, we give a detailed proof of the following inequality due to Gromov: $\kcw(M) \le c \acw(M)$, where $c$ is a dimensional constant.
\end{abstract}

\maketitle
\section{Two invariants associated with the positive scalar curvature}
The famous Lichnerowicz theorem asserts that for a closed spin manifold $M$ if the $\ahg$-genus of $M$ does not vanish, $M$ cannot carry a Riemannian metric $g$ with the scalar curvature $\kappa_g > 0$.
An interesting generalization of this theorem is to find a ``quantitative'' version of it.
Namely, under what conditions, can we obtain an upper bound for $\inf_M \kappa_g$ holding for any Riemannian metric $g$ on $M$?
To investigate this problem, Gromov~\cite{Gromov_1996aa} formulates an invariant called $\rK$-cowaist.\footnote{In~\cite{Gromov_1996aa}, $\rK$-cowaist was called $\rK$-area. But recently, Gromov~\cite{Gromov_2019fv} suggests that $\mathrm{K}$-cowaist should be a more proper name for this concept.}
Let us recall the definition of this invariant.

Let $M$ be a closed connected oriented smooth Riemannian manifold of even dimension.
Let $E\to M$ be a Hermitian vector bundle with a Hermitian connection $\nabla^E$ and $\opR^E$ be the curvature of $\nabla^E$.
Note that for any $x\in M$ and $\alpha,\beta \in \rT_x M$, $\opR^E(\alpha \wedge \beta) \in \End{E_x}$.
We define a norm on $\opR^E$ in the following way,
\begin{equation*}
  \|\opR^{E}\| = \sup_{x\in M} \sup_{\substack{\alpha,\beta\in \rT_x{M} \\ \alpha \perp \beta,\ |\alpha\wedge \beta|=1}}|\opR^{E}(\alpha\wedge \beta)|.
\end{equation*}
We are interested in the following class of unitary vector bundles over $M$,
\begin{equation}
  \label{eq:cn}
  \exists\; i_1,\cdots,i_{l} \in \bZ_{> 0} \text{ such that } \int_M \prod_{k=1}^l\opc_{i_k}(E) \neq 0,
\end{equation}
where $\opc_{i_k}(E)$ is the $i_k$-th Chern class of $E$.
The $\rK$-cowaist of $M$ is defined to be
\begin{equation*}
    \kcw(M)=\sup\{\|\opR^E\|^{-1}| E \text{ satisfies (\ref{eq:cn})}\}.
\end{equation*}
In Gromov's words, the condition on the Chern numbers of $E$ ensures that $E$ is ``homologically non-trivial''.
In view of Lichnerowicz's theorem, it is also natural to replace the condition (\ref{eq:cn}) of $\rK$-cowaist with the following condition,
\begin{equation}
  \label{eq:ahg}
  \int_M \ahg(M)\ch(E) \neq 0.
\end{equation}
And the $\ahg$-cowaist\footnote{In~\cite{Cecchini_2021sc}, this invariant is called $\ahg$-area.
We choose to rename it as $\ahg$-cowaist to be incoordination with the name of $\rK$-cowaist.} of $M$ is defined to be
\begin{equation*}
  \acw(M)=\sup\{\|\opR^E\|^{-1}| E \text{ satisfies (\ref{eq:ahg})}\}.
\end{equation*}

Using the $\rK$-cowaist, Gromov gives the following quantitative version of Lichnerowicz's theorem.
Let $(M,g)$ be a spin Riemannian manifold of even dimension.
\begin{equation}
  \label{eq:g-est}
  \inf_M \kappa_g\le C (\kcw(M))^{-1},
\end{equation}
where $C$ is a constant only depending on the dimension of $M$. In~\cite{Gromov_1996aa}, although the definition of $\ahg$-cowaist does not appear explicitly, the proof of (\ref{eq:g-est}) uses it and consists of two steps.
The first step is to show (\ref{eq:g-est}) with $\kcw(M)$ replaced by $\acw(M)$. The second step is proving the following comparison results between $\kcw(M)$ and $\acw(M)$.
\begin{theorem}
  \label{thm:com}
  Let $M$ be a closed Riemannian metric manifold $M$ of even dimension.
  There is a constant $c$ depending only on the dimension of $M$ such that
  \begin{equation}
    \label{eq:klea}
    \kcw(M) \le c \acw(M).
  \end{equation}
\end{theorem}
In this note, we would like to clarify some arguments used in~\cite{Gromov_1996aa} to prove Theorem~\ref{thm:com}, see Remark~\ref{rk:gromov}.
In other words, we give a detailed proof of this theorem following Gromov's method.

Before proving Theorem~\ref{thm:com}, we note that it has the following corollary.
\begin{corollary}
  \label{cor:inf}
  If $\kcw(M) = +\infty$, then $\acw(M) = +\infty$.
\end{corollary}
This corollary, as well as its generalization for manifolds with boundary or non-compact manifolds, has found application in the literature, see~\cite{Cecchini_2021sc,Su_2021ma}.

This paper is organized as follows.
In Section~\ref{sec:pf-com}, after some algebraic preliminaries, we prove Theorem~\ref{thm:com}.
In Section~\ref{sec:so-re}, we give two remarks about Theorem~\ref{thm:com}, about the its generalization on non-closed manifolds and the reverse direction inequality of (\ref{eq:klea}) respectively.

\section{Proof of Theorem~\ref{thm:com}}
\label{sec:pf-com}
In the rest of this paper, we assume that the dimension of $M$ is even.
Using the definition of $\rK$-cowaist, we can see that the following result implies (\ref{eq:klea}).
\begin{proposition}
  \label{prop:klea}
  Fix a positive number $m_0$.
  Let $E$ be a Hermitian vector bundle\footnote{In the following, every Hermitian vector bundle carries a Hermitian connection implicitly.} over $M$ satisfying (\ref{eq:cn}) and
  \begin{equation}
    \label{eq:ba0}
    \Vert\opR^E\Vert^{-1} \ge m_0.
  \end{equation}
  Then there exists a Hermitian vector bundle $E'$ over $M$ satisfying (\ref{eq:ahg}) and
  \begin{equation}
    \label{eq:ca}
    \Vert\opR^{E'}\Vert^{-1} \ge c m_0
  \end{equation}
  where $c$ is constant only depending on the dimension of $M$.
\end{proposition}

\subsection{Two algebraic lemmas}
\begin{definition}
  \label{def:ad-f}
  Let $\cV$ be the category of the Hermitian vector bundles over a fixed manifold.We call a functor $J:\cV \times \cdots \times \cV \rightarrow \cV$ \emph{admissible} if $J$ is a finite composition of following functors:
  \begin{enumerate}
  \item $I$, the identify functor of $\cV$;
  \item $E \rightarrow \bC^k$, where $\bC^k$ is the trivial bundle;
  \item $E \rightarrow E'$, where $E'$ is the dual bundle of $E$;
  \item $E \rightarrow \wedge^k E$, where $\wedge^k E$ is the $k$-th wedge product bundle of $E$;
  \item $E, F \rightarrow E\oplus F$, the direct sum;
  \item $E, F \rightarrow E\otimes F$, the tensor product.
  \end{enumerate}
  Let $E,F$ (resp.\ $E_1,\cdots,E_k,F$) be Hermitian vectors bundles.
We say that \emph{$F$ is constructed from $E$ (resp.\ $E_1,\cdots,E_k$) in an admissible way}, if there exists an admissible functor $J$ and an isomorphism between $F$ and $J(E)$ (resp.\ $J(E_1,\cdots,E_k)$), which preserves the metric and the connection.
\end{definition}
\begin{remark}
  Note that the six simple operations on Hermitian vector bundles listed in Definition~\ref{def:ad-f} exist on every manifold.
  Therefore, for two manifolds $M$ and $N$, there is a natural bijection between the sets of admissible functors defined by $M$ and $N$.
  In this sense, the definition of admissible functors does not depend on the ambient manifolds.
  We would like to remark that in the following, when we say that a constant depends only on an admissible functor $J$, we precisely mean that the constant does not depend on the choice of ambient manifold defining $J$.
\end{remark}
For example, $\mathbb{C}^n$, $E'$, $E\oplus E$ and $(\otimes^n E) \otimes (\wedge^k E)$ are all constructed from $E$ in an admissible way.
\begin{proposition}
  If $J$ is an admissible functor, there is a constant $C_0$ such that
  \begin{equation}
    \label{eq:c0}
    \Vert \opR^{J(E)} \Vert \le C_J \Vert \opR^{E} \Vert
  \end{equation}
  holds for any Hermitian vector bundle $E$, where $C_0$ is a constant depending only on $J$.
\end{proposition}
\begin{proof}
  By the definition of the admissible functor, we only need to check (\ref{eq:c0}) for the functors listed in Definition~\ref{def:ad-f}.
  We check (\ref{eq:c0}) for the tensor product as an example.

  Let $V,W$ be two Hermitian vector spaces and $A,B$ be endomorphisms on $V,W$ respectively.
  We note that
  \begin{equation*}
    (A \otimes \id)^*(A \otimes \id) = (A^* \otimes \id)(A \otimes \id) = A^*A \otimes \id.
  \end{equation*}
  By the relation between the singular values and the matrix norm, the above equality implies that
  \begin{equation*}
    \Vert A \otimes \id \Vert = \Vert A \Vert.
  \end{equation*}
  As a result, we have
  \begin{equation}
    \label{eq:m-inq}
    \Vert A \otimes \id + \id \otimes B\Vert \le \Vert A \Vert + \Vert B \Vert.
  \end{equation}

  By the definition of the tensor product of connections, for Hermitian vector bundles $E,F$, we have
  \begin{equation*}
    \opR^{E\otimes F} = \opR^E \otimes \id + \id \otimes \opR^F.
  \end{equation*}
  Then, the needed inequality
  \begin{equation*}
    \Vert \opR^{E\otimes F} \Vert \le \Vert \opR^E \Vert + \Vert \opR^F \Vert
  \end{equation*}
  follows from (\ref{eq:m-inq}) immediately.
\end{proof}

After Gromov, we use the following algebraic lemma to prove Proposition~\ref{prop:klea}.
\begin{lemma}
  \label{lm:alg}
  Fix $N\in \bZ_{+}$.
  There is a finite set of admissible functors $\cC_N = \aset{J_i}$
satisfying the following property.
  For any positive integer $K\le N$ and any partition of $K$ by positive integers, $K = \sum_{l = 1}^{\nu} a_l$, there exist $\lambda_i\in \bQ$ such that
  \begin{equation}
    \label{eq:cal}
    \prod_{l=1}^{\nu} \opc_{a_l}(E) = \sum_i \lambda_i\ch_K(J_i(E))
  \end{equation}
  holds for any Hermitian vector bundle $E$, where $\ch_K$ denotes the degree $2K$ component of the Chern character.
\end{lemma}

This lemma is just a restatement of~\cite[p.~36, Trivial Algebraic Lemma]{Gromov_1996aa}.
For readers' convenience, we also provide a proof for this lemma in Subsection~\ref{sub:trivial}.

To state another algebraic lemma, let $\psi_k$ denote the $k$-th Adams operation of a complex vector bundle.
It is known that $\psi_k(E)$ can be constructed from $E$ in an admissible way,\footnote{More precisely, $\psi_k(E)$ is a virtual bundle $F_1 - F_2$, where $F_1, F_2$ can be constructed from $E$ in an admissible way respectively.}~\cite{Adams_1962aa,Atiyah_1969gr}.
The Chern character of $E$ and $\psi_k(E)$ have the following relation.
\begin{equation}
  \label{eq:ch-psi}
  \ch(\psi_k(E)) = \sum_{i=0}^{\dim M/2} \ch_i(E) k^i.
\end{equation}
The second algebraic lemma we need is as follows.
\begin{lemma}
  \label{lm:ch}
  Let $E$ be a Hermitian vector bundle over $M$.
  If for $1\le k \le \dim M /2 + 1$,
  \begin{equation}
    \label{eq:ae}
    \int_M \ahg(M) \ch(\psi_k(E)) = 0,
  \end{equation}
  then $\int_M \ch(E) = 0$.
\end{lemma}
\begin{remark}
  \label{rk:gromov}
  In~\cite[p.~36]{Gromov_1996aa}, Gromov uses a result similar to Lemma~\ref{lm:ch}.
  But the condition of the lemma, (\ref{eq:ae}), is replaced with the condition that for all $k$,
  \begin{equation*}
    \int_M \ahg(M) (\ch(E))^k = 0.
  \end{equation*}
  In our opinion, it is not very straightforward to see why such a condition also yields the same conclusion of Lemma~\ref{lm:ch}.
  Our main motivation to write this note is to clarify this point.
\end{remark}

\begin{proof}[Proof of Lemma~\ref{lm:ch}]
  Let $\dim M = 2n$.
  Denote the degree $2l$ component of $\ahg(M)$ by $\ahg_l(M)$.
  Note that if $l$ is an odd number, $\ahg_l(M) = 0$.
  For $0\le i \le n$, let
  \begin{equation*}
    a_i \coloneqq \int_M \ahg_{n-i}(M) \ch_i(E)
  \end{equation*}
  and $\mathbf{a} \coloneqq [a_0, a_1, \cdots, a_n]^{\rT}$.
  Then by (\ref{eq:ch-psi}), the condition (\ref{eq:ae}) gives
  \begin{equation}
    \label{eq:ai}
    \bfL \mathbf{a} = \mathbf{0},\quad \text{where}\quad
    \bfL \coloneqq
    \begin{bmatrix}
      1 & 1 & 1 & \cdots & 1 \\
      1 & 2 & 2^2 & \cdots & 2^n \\
      1 & 3 & 3^2 & \cdots & 3^n \\
      \vdots& \vdots &  \vdots  & \ddots & \vdots \\
      1 & n+1 & (n+1)^2 & \cdots & (n+1)^n
    \end{bmatrix}.
  \end{equation}
  Since $\bfL$ is a Vandermonde matrix, its determinant is $\det{\bfL} = \prod_{1\le i< j \le n+1}(j-i) \neq 0$.
  As a result, (\ref{eq:ai}) implies $\mathbf{a} = 0$, which means $\int_M \ch(E) = a_n = 0$.
\end{proof}

\subsection{Proof of Proposition~\ref{prop:klea}}
We use Lemma~\ref{lm:alg} and Lemma~\ref{lm:ch} to show this proposition.

Take the constant $N$ in Lemma~\ref{lm:alg} to be $\dim M/2$.
Since $\cC_N$ is a finite set, by the estimate (\ref{eq:c0}), we have
\begin{equation}
  \label{eq:est-c}
  \sup_{J\in \cC_N} \Vert \opR^{J(E)} \Vert \le A_N \Vert \opR^{E} \Vert,
\end{equation}
where $A_N = \sup_{J\in \cC_N} C_{J}$ is a constant depending only on $N$.

Since there is a nonvanishing Chern number for $E$, due to Lemma~\ref{lm:alg}, we can find $J_1 \in \cC_N$ and $E_1\coloneqq J_1(E)$ such that
\begin{equation*}
  \int_M \ch(E_1) = \int_M\ch_N(E_1) \neq 0.
\end{equation*}
By (\ref{eq:ba0}) and (\ref{eq:est-c}), we also have
\begin{equation*}
  \label{eq:est-e1}
  \Vert \opR^{E_1} \Vert \le A_N /m_0.
\end{equation*}
Then by Lemma~\ref{lm:ch}, there exists an integer $1 \le k_0 \le \dim M/2 + 1$ satisfying
\begin{equation*}
  \int_M \ahg(M) \ch(\psi_{k_0}(E_1)) \neq 0.
\end{equation*}
As we have noted, there exist admissible functors $G_{k_0}^1,G_{k_0}^2$ such that $\psi_{k_0}(E_1)$ is a virtual bundle $G_{k_0}^1(E_1) - G_{k_0}^2(E_1)$.
Therefore, the above inequality implies that at least one of $G_{k_0}^1(E_1)$ and $G_{k_0}^2(E_1)$, say $G_{k_0}^1(E_1)$, satisfies
\begin{equation*}
  \int_M \ahg(M) \ch(G_{k_0}^1(E_1)) \neq 0.
\end{equation*}
Using (\ref{eq:c0}) again, we have
\begin{equation*}
  \Vert \opR^{G_{k_0}^1(E_1)} \Vert \le C_{k_0} \Vert \opR^{E_1} \Vert \le \max_{1 \le k_0 \le N +1}C_{k_0} A_N/m_0,
\end{equation*}
where $C_{k_0}  = \max(C_{G_{k_0}^1}, C_{G_{k_0}^2})$.
Since the functors $G_{k_0}^1, G_{k_0}^2$ depend only on $k_0$, $C_{k_0}$ is also a constant depending only on $k_0$.
Therefore, $G_{k_0}^1(E_1)$ satisfies (\ref{eq:ca}) and the proof of Proposition~\ref{prop:klea} is finished.

\subsection{Proof of Lemma~\ref{lm:alg}.}
\label{sub:trivial}
As we will see, the proof of Lemma~\ref{lm:alg} is similar to the proof of Lemma~\ref{lm:ch} in some sense.

Firstly, by~\cite[Lemma~2.1.6]{Gilkey_1995aa}, we know that $\opc_i(E)$ can be represented by a homogeneous polynomial of $\aset{\ch_j(E)}$ and such a polynomial is independent of the choice of $E$.
Therefore, we only need to prove Lemma~\ref{lm:alg} with (\ref{eq:cal}) replaced by the following equality.
\begin{equation}
  \label{eq:chal}
  \prod_{l=1}^{\nu} \ch_{a_l}(E) = \sum_i \lambda_i\ch_K(J_i(E)).
\end{equation}
We will prove this result by using induction on $\nu$.
More specifically, for each $\nu$, we construct a finite set $\cC_N^\nu$ inductively such that by choosing $J_i\in \cC_N^\nu$, (\ref{eq:chal}) holds for $\nu$.

As a first step, set $\cC_N^1 = \aset{I}$.
By definition, since $I(E) = E$, (\ref{eq:chal}) holds for $\nu = 1$.
Supposing that for $\nu\ge 1$, we have constructed a finite set $\cC_N^{\nu}$ such that by choosing $J_i\in \cC_N^\nu$, (\ref{eq:chal}) holds for $\nu$.
We will construct $\cC_N^{\nu+1}$ using $\cC_N^{\nu}$.

For $K\le N$, we take an abitrary $\nu+1$ partition of $K$, $K= a_0 + \cdots + a_{\nu}$, $a_i\in \bZ_+$.
Let $K_1 \coloneqq a_1 + \cdots + a_{\nu} \le N$.
Using the induction assumption for $\prod_{l=1}^{\nu} \ch_{a_l}(E)$, we can find $\lambda_j\in \bQ$ such that
\begin{equation}
  \label{eq:ch-nu}
  \ch_{a_0}(E)\prod_{l=1}^{\nu} \ch_{a_l}(E) = \sum_j \lambda_j\ch_{a_0}(E)\ch_{K_1}(J_j(E)),
\end{equation}
where $J_j\in \cC_N^{\nu}$.
To deal with the r.h.s.\ of (\ref{eq:ch-nu}), we use the following result, which is a variation of the $\nu=2$ case of Lemma~\ref{lm:alg}.
\begin{lemma}
  \label{lm:alg2}
  Let $i,j\in \bZ_{\ge1}$.
  There is a finite set of admissible functors $\aset{G_l}$ and $\lambda_l\in \bQ$ such that
  \begin{equation*}
    \ch_i(F_1)\ch_j(F_2) = \sum_l \lambda_l\ch_{i+j}(G_l(F_1,F_2))
  \end{equation*}
  holds for any two Hermitian vector bundles $F_1,F_2$.
\end{lemma}
\begin{proof}
  Take $H_l = \psi_l(F_1) \otimes F_2$.
  Then, there exist admissible functors $G^1_l,G^2_l$ such that $H_l = G_{l}^1(F_1,F_2)-G_{l}^2(F_1,F_2)$.
  Set $r = i+ j$.
  By (\ref{eq:ch-psi}), we have
  \begin{multline*}
    \ch_r(G_{l}^1(F_1,F_2))-\ch_r(G_{l}^2(F_1,F_2))= \ch_r(H_l) \\
    = \sum_{a = 0}^{r} \ch_a(\psi_l(F_1))\ch_{r-a}(F_2)
    = \sum_{a = 0}^{r} \ch_a(F_1) \ch_{r-a}(F_1) l^a.
  \end{multline*}
  Now, by choosing $l\in \aset{1,\cdots,r+1}$, we have an invertible Vandermonde matrix as in the proof of Lemma~\ref{lm:ch}.
Therefore, the set $\aset{G_{1}^1,G_{1}^2,\cdots,G_{r+1}^1,G_{r+1}^2}$ satisfies the requirement in the lemma.
\end{proof}

By Lemma~\ref{lm:alg2}, for any $J\in \cC^{\nu}_N$ and fixing $a_0,K_1$ such that $ a_0\ge 1$, $K_1 \ge 1$ and $a_0 + K_1 \le N$, we can construct finitely many admissible functors $G_{J,a_0,K_1,l}$ such that
\begin{equation}
  \label{eq:ch-nv1}
  \ch_{a_0}(E)\ch_{K_1}(J(E)) = \sum_l \lambda'_{l} \ch_{a_0+ K_1}(G_{J,a_0,K_1,l}(E)).
\end{equation}
Set
\begin{equation*}
  \cC_N^{\nu+1} = \cC_N^{\nu} \bigcup \aset{G_{J,a_0, K_1,l}|{J\in\cC^{\nu}_N, a_0\ge 1, K_1 \ge 1, a_0+K_1 \le N}}.
\end{equation*}
By (\ref{eq:ch-nu}) and (\ref{eq:ch-nv1}), by choosing $J_i\in \cC_N^{\nu+1}$, (\ref{eq:chal}) holds for $\nu+1$.
The proof of Lemma~\ref{lm:alg} is finished.

\section{Two remarks about Theorem~\ref{thm:com}}
\label{sec:so-re}
\subsection{Theorem~\ref{thm:com} for more general manifolds}
Till now, we assume that the manifold $M$ is closed.
Now, we comment briefly on how to extend Theorem~\ref{thm:com} to compact manifolds with boundary or non-compact manifolds.

In~\cite{Cecchini_2021sc,Gromov_1996aa,Su_2021ma}, the authors define the $\rK$-cowaist and $\ahg$-cowaist on compact manifolds with boundary or non-compact manifolds.
The idea behind these two kinds of generalization is the same.
Let us recall the definition in~\cite{Cecchini_2021sc} as an example.

Let $M$ be a compact manifold with boundary.
To take the effect of boundary into consideration, we choose a pair of Hermitian bundles $E,F$, called a \emph{compatible pair}, such that there is an isomorphism between $E$ and $F$ near $\partial M$ which preserves the metric and the connection on $E$ and $F$.\footnote{In~\cite{Cecchini_2021sc}, $E,F$ are called the admissible pair.
  We rename it to avoid the possible ambiguities with the admissible functors used in this paper.
}
To replace the condition (\ref{eq:cn}), we use the condition that there exists a polynomial $p$ of Chern forms such that
\begin{equation}
  \label{eq:cn1}
  \int_M (p(\opc_0(E),\opc_1(E),\cdots) - p(\opc_0(F),\opc_1(F),\cdots)) \neq 0.
\end{equation}
Accordingly, to replace the condition (\ref{eq:ahg}), we use the condition that
\begin{equation}
  \label{eq:ahg1}
  \int_M \ahg(M)(\ch(E) - \ch(F)) \neq 0.
\end{equation}
Then the $\rK$-cowaist (resp.\ $\ahg$-cowaist) of $M$ is the supremum of $\|\opR^{E \oplus F}\|^{-1}$ with respect to all compatible pairs $E,F$ satisfying (\ref{eq:cn1}) (resp.\ (\ref{eq:ahg1})).

For the case that $M$ is non-compact, these two definitions remain valid if we modify the definition of the compatible pair a little.
Namely, in this non-compact case, we require that the isomorphism between $E$ and $F$ is defined outside a compact set of $M$.
Note that for different compatible pairs on $M$, the compact set may vary.

The method to show Theorem~\ref{thm:com} in fact also works for these more general cases.
The key point is that the equality (\ref{eq:ch-psi}), although we treat it as a cohomological equality in Section~\ref{sec:pf-com}, holds at the differential form level.
To check this fact, we can use the explicit construction of $\psi_k(E)$ given in~\cite[\S~4]{Adams_1962aa}.
Then one can use the same arguments to show that Lemma~\ref{lm:alg} and Lemma~\ref{lm:ch} still hold with the vector bundles replaced by the compatible pairs.

\subsection{The reverse direction inequality of (\ref{eq:klea})}
As another remark for Theorem~\ref{thm:com}, we would to like to point out that an inequality like (\ref{eq:klea}) in the reverse direction does not hold in general.

Let $N$ be a 4-dimensional simply connected closed manifold with $\int_N\ahg(N) \neq 0$ (e.g.\ a K$3$ surface) and $\bS^2(R)$ be the standard 2-dimensional sphere with the radius $R$.
Besides, we denote the Hopf bundle over $\bS^2(R)$ by $H$.

In~\cite[\S~4{$\frac{1}{4}$}]{Gromov_1996aa}, Gromov shows that $\kcw(N) < +\infty$.
In fact, by checking the proof of this result, we know that there exists a constant $a$ depending on $N$ such that for any Hermitian vector bundle $L$ over $N$ with $\Vert \opR^L\Vert < a$, $L$ must be a topological trivial bundle (with a possible nontrivial metric and connection).
Furthermore, the same proof implies that for any Hermitian vector bundle $E$ over $N \times \bS^2(R)$ such that $\Vert \opR^E\Vert < a$, $E$ must be topologically isomorphic to a pullback bundle from $\bS^2(R)$, which implies that $E$ cannot satisfy (\ref{eq:cn}).
As a result, we know that
\begin{equation}
  \label{eq:ns1}
  \kcw(N \times \bS^2(R)) \le a^{-1}.
\end{equation}

On the other hand, we denote the pullback bundle of $H$ over $N \times \bS^2$ by $\bar{H}$.
We have
\begin{equation*}
  \int_{N \times \bS^2(R)} \ahg(N \times \bS^2(R)) \ch(\bar{H}) = \int_{N} \ahg{(N)}\int_{\bS^2(R)}\opc_1(H) \ne 0,
\end{equation*}
that is, $\bar{H}$ satisfies (\ref{eq:ahg}).
However, by direct calculation,
\begin{equation*}
  \Vert\opR^{\bar{H}}\Vert = \frac{1}{2R^2}.
\end{equation*}
As a result,
\begin{equation}
  \label{eq:ns2}
  \acw(N \times \bS^2(R)) \ge 2R^2.
\end{equation}
Combining (\ref{eq:ns1}) and (\ref{eq:ns2}), we know that (\ref{eq:klea}) in the reverse direction does not hold in general.

\section{Acknowledgments}
The author would like to thank Prof. Guangxiang Su for helpful discussion about the content of this paper and the anonymous referee for reading the paper carefully and the very inspiring suggestions.

\providecommand{\bysame}{\leavevmode\hbox to3em{\hrulefill}\thinspace}
\providecommand{\MR}{\relax\ifhmode\unskip\space\fi MR }
\providecommand{\MRhref}[2]{\href{http://www.ams.org/mathscinet-getitem?mr=#1}{#2}
}
\providecommand{\href}[2]{#2}

\end{document}